\documentclass[12pt,reqno]{amsart}
\usepackage{amsmath,amssymb,amsthm}
\usepackage[margin=1in]{geometry}
\usepackage{hyperref}
\usepackage{cleveref}

\title[Some Central Digraphs]{\bfseries Some Central Digraphs and their Automorphisms}

\author{Jacob Siehler}
\address{Gustavus Adolphus College\\
800W College Avenue\\
St Peter, MN 56082}
\email{jsiehler@gustavus.edu}



\theoremstyle{plain}
\newtheorem{theorem}{Theorem}
\newtheorem{lemma}[theorem]{Lemma}
\newtheorem{corollary}[theorem]{Corollary}
\newtheorem{proposition}[theorem]{Proposition}

\theoremstyle{definition}
\newtheorem{definition}[theorem]{Definition}
\newtheorem{example}[theorem]{Example}

\theoremstyle{remark}

\def\isomorphic{\cong}
\def\Int{\mathbb{Z}}
\def\Zmod#1{\Int/#1\Int}
\def\Zk{\Zmod{k}}
\def\Zpp{\Zmod{p^2}}
\def\FieldF{\mathbb{F}}
\def\Zmodmult#1{(\Int/#1\Int)^{\times}}
\def\divides{\mid}
\def\gname{metacyclic central digraph}
\def\gnames{\gname\relax s}
\def\kname{size}
\def\cname{multiplicative order}
\def\wname{unit}
\DeclareMathOperator{\Aut}{Aut}
\DeclareMathOperator{\Iso}{Iso}
\DeclareMathOperator{\im}{im}
\DeclareMathOperator{\quo}{quo}
\DeclareMathOperator{\rem}{rem}
\DeclareMathOperator{\car}{car}
\DeclareMathOperator{\bor}{bor}
\DeclareMathOperator{\eulerphi}{\varphi}

\begin{document}

\begin{abstract}
We define a family of finite directed graphs $G(k,c,w)$ on vertex set $\Zk\times \Zk$, where $w\in\Zk^\times$ has multiplicative order $c$, which divides $k$.  We show that each $G(k,c,w)$ is a central digraph (its adjacency matrix squares to the all-ones matrix).  We deduce explicit descriptions of the graph automorphism groups for graphs of type $G(k,1,1)$; $G(k,2,-1)$ when $k$ is even; and $G(p^2, p, 1+p)$, when $p$ is prime. 
\end{abstract}
\maketitle


\section{Introduction and Main Results}

A \textit{central digraph} is a directed graph characterized by the property that the square of its adjacency matrix is a matrix of all ones.  Equivalently, for any vertices $x$ and $z$ in the graph (not necessarily distinct), there is a unique vertex $y$ such that $x\to y$ and $y\to z$ are edges in the graph.  We will call $y$ a \textit{link} from $x$ to $z$ in this case.

There is a one-to-one correspondence between central digraphs and the algebraic structures known as \textit{central groupoids}, detailed in Knuth~\cite{knuth1970notes}.  We will take the graph-theoretic point of view, although due to the correspondence, every result here can be interpreted as a result about central groupoids as well.  Interest in these structures seems to have originated with the problem, posed by Hoffman~\cite{hoffman1967}, of enumerating $0/1$ matrices that square to a matrix of ones.  

In this paper, we define a parametrized family of central digraphs, using elementary number theory.  We show that distinct parameters give rise to distinct (non-isomorphic) digraphs, and exhibit some values of the parameters for which it is possible to determine the automorphism group of the digraph.  

Before defining the graphs in question, let us lay out a few elementary properties of (finite) central digraphs in general, all of which are succinctly established both in \cite{knuth1970notes} and in \cite{curtis2004central}.  If $G$ is a central digraph, then
\begin{enumerate}
\item The number of vertices in $G$ must be a square; say $G$ has $k^2$ vertices.
\item $G$ has $k^3$ edges; each vertex has in-degree $k$ and out-degree $k$.
\item $G$ has exactly $k$ vertices with loops.  These may be called \textit{idempotent vertices}; the term ``idempotent'' comes from the central groupoid interpretation.
\end{enumerate}

The following definition introduces our objects of study.

\begin{definition}\label{def:GraphDef}
Let $k$ be a positive integer, and let $w$ be an element of multiplicative order $c$ in $\Zmod{k}$, where $c\divides k$.  Define graph $G\left(k,c,w\right)$ as follows:  the vertex set is the cartesian product $\Zmod{k}\times\Zmod{k}$, and the edge set consists of all directed edges $\left(p,i\right)\to\left(p+w^ji,j\right)$, where $p$, $i$, and $j$ are elements of $\Zmod{k}$.
\end{definition}

\textit{Note.} In the definition, $j$ is an element of $\Zk$, but it is used as an exponent in the expression $p+w^ji$.  This would ordinarily be ill-defined, but the condition that $c\divides k$ ensures that $w^j$ is well-defined, independent of the specific integer chosen to represent the class of $j$.

We propose the name \textit{\gnames} for these graphs, due to the resemblance of the edge rule to the operation in a semidirect product $\Zk\rtimes \Zk$.  The parameter $k$ in $G(k,c,w)$ will be called the \textit{\kname} and $c$ will be called the \textit{\cname} of the graph.  Parameter $w$ will be called the \textit{\wname} of the graph.

In the rest of this section, we state our main results.  The proofs will follow in the subsequent sections.

\begin{theorem}\label{thm:CorrectConstruction}
The graph $G\left(k,c,w\right)$ defined above is a central digraph; that is, for every $\left(p,i\right)$ and every $\left(q,j\right)$ in $\Zmod{k}\times \Zmod{k}$, there is a unique link from $\left(p,i\right)$ to $\left(q,j\right)$ in 
$G\left(k,c,w\right)$.  
\end{theorem}

For any value of $k$, one can use the parameters $c=1$ and $w=1$ to construct a central digraph.  We shall not prove it, but this choice of parameters corresponds to what Knuth~\cite{knuth1970notes} calls the {\it natural} central groupoid on $k^2$ elements.  Other parameters do give rise to distinct graphs.  For example, for any even integer $k>2$, one can choose $c=2$ and $w=-1$ to obtain a simple example of a central digraph which is distinct from the ``natural'' one.

When $w$ is a unit mod $k$, multiplication by $w$ is a permutation of $\Zk$.  The cycle type of this permutation is reflected in the structure of the corresponding \gname, as the following proposition makes precise.

\begin{theorem}\label{thm:DistinctGraphs}
If $G\left(k,c_1,w_1\right) \isomorphic G\left(k,c_2,w_2\right)$, then multiplication by $w_1$ and multiplication by $w_2$ are permutations of the same cycle type.  In particular, $c_1=c_2$.
\end{theorem}

Note that the parameters $k$ and $c$ alone are not sufficient to determine the graph up to isomorphism; the choice of \wname\ may matter.  For example, $G(8,2,-1)$ and $G(8,2,5)$ are non-isomorphic graphs.  This is guaranteed by the preceding theorem since multiplication by $5$ permutes $\Zmod{8}$ as $\left(1\ 5\right)\left(3\ 7\right)$, whereas multiplication by $-1$ is $\left(1\ 7\right)\left(2\ 6\right)\left(3\ 5\right)$ -- different cycle types.  One can also check with a computer that the adjacency matrix of the former graph has rank $12$, while the latter has rank $14$.  

I \emph{conjecture} that the converse is also true -- if multiplication by $w_1$ and by $w_2$ have the same cycle type (as permutations of $\Zk$), then $G(k,c,w_1) \isomorphic G(k,c,w_2)$.  This is true in many special cases, but so far I do not have a general proof.

One can see immediately from the definition that the vertices of the form $(x,0)$ are the idempotent vertices in $G(k,c,w)$.  Any graph automorphism must permute these vertices among themselves. The following theorem establishes the importance of this action to determining the graph automorphism group.  

\begin{theorem}\label{thm:AutosInSk}
Any automorphism of $G\left(k,c,w\right)$ is determined by its action on the set of $k$ idempotent vertices, and the automorphism group of $G\left(k,c,w\right)$ is isomorphic to a subgroup of the symmetric group $S_k$.
\end{theorem}

The simplest choice of parameters leads to a graph with an easy-to-understand, and maximal, automorphism group.

\begin{theorem}\label{thm:AutoGroup1}
For any positive integer $k$, the automorphism group of $G(k,1,1)$ is isomorphic to $S_k$.
\end{theorem}

In general, though, the automorphism groups are much smaller than this and show a wide variety of structures.  The following two theorems give specific families of examples, of different complexity.

\begin{theorem}\label{thm:AutoGroup2}
For any positive, even integer $k$, the automorphism group of $G(k,2,-1)$ is isomorphic to $\text{Aff}\left(\Zk\right)$,
that is, $\Zk\rtimes \Zmodmult{k}$.
\end{theorem}

\begin{theorem}\label{thm:AutoGroup3}
If $p$ is prime, then the automorphism group of $G(p^2, p, 1+p)$ is a nonabelian group of order $p^{p+1}\cdot(p-1)$.  The group is a semidirect product, with
\begin{enumerate}
\item A normal, elementary abelian subgroup of order $p^p$, and
\item A complementary subgroup of order $p\cdot(p-1)$ isomorphic to $\text{Aff}\left(\FieldF_p\right)$.
\end{enumerate}
\end{theorem}

For the simplest example of Theorem~\ref{thm:AutoGroup3}, using $p=2$, we find that $G(4,2,3)$ is a graph on $16$ vertices with an automorphism group which is nonabelian, of order $8$.  Furthermore, it has a normal $Z_2\times Z_2$ subgroup, so it must be the $8$-element dihedral group.

\section{Proofs of Theorems~\ref{thm:CorrectConstruction} and \ref{thm:DistinctGraphs}}

Before proving that Definition~\ref{def:GraphDef} does in fact define a central digraph, we will establish a few purely number-theoretic results.  The reader might prefer to skip ahead to the short proof of Theorem~\ref{thm:CorrectConstruction} and return to the next few lemmas afterward if the details are of interest.

\begin{lemma}\label{lem:SmallestPrime}
Suppose $w$ has multiplicative order $c$ in $\Zk$, that $c\divides k$, and $p$ is the smallest prime divisor of $c$.  Then $$w^{\left(p^{i-1}\right)}\equiv 1\pmod{p^i}$$ for each $i\ge 1$.
\end{lemma}

\begin{proof}
This proceeds by induction on the exponent $i$.  For the base case, note that
$w^c\equiv 1\pmod{k}$
implies
$w^c\equiv 1\pmod{p},$
since $p\divides k$.
By Fermat's Little Theorem, we also know that 
$w^{p-1}\equiv 1\pmod{p}.$
Since $p$ is the smallest prime divisor of $c$, it follows that $\gcd(c,p-1)=1$, hence $w\equiv 1\pmod{p}$.

Now, inductively, suppose $w^{\left(p^{i-1}\right)}\equiv 1\pmod{p^i}$ for some $i$ with $1\le i$.  That means
$$w^{\left(p^{i-1}\right)} =  1+p^iq$$
for some integer $q$.  Taking the $p$-th power of both sides and applying the binomial theorem on the right,
$$w^{\left(p^i\right)} = 1+\binom{p}{1}p^iq+\text{terms divisible by $p^{i+1}$},$$
hence $w^{\left(p^i\right)}\equiv 1\pmod{p^{i+1}}$, completing the inductive step.

By induction, the conclusion holds for every $i\ge 1$.
\end{proof}

\begin{lemma}\label{lem:DivisibilitySPF}
Suppose $w$ has multiplicative order $c$ in $\Zk$, and $c\divides k$.  Let $p$ denote the smallest prime divisor of $c$, factoring $c$ as $p^a\cdot m$ with $p\nmid m$.
 Let $g$ denote any element of $\Zk$.  If $g\left(1-w^d\right)\equiv d\pmod{k}$, then $d\equiv 0\pmod{p^a}$.
\end{lemma}

\begin{proof}
This also proceeds by induction on the power of $p$.  By Lemma~\ref{lem:SmallestPrime}, $w\equiv 1\pmod{p}$, so reducing the main hypothesis mod $p$ tells us $0\equiv d\pmod{p}$, establishing a base case.

Now suppose inductively that $d\equiv 0\pmod{p^i}$ for some $i\ge
1$.  Say $d=p^iq,$ where $q\in\Int$.  Reducing the main hypothesis
mod $p^{i+1}$ gives 
$$g\left(1-w^{p^iq}\right)\equiv p^iq\pmod{p^{i+1}}.$$

Lemma~\ref{lem:SmallestPrime} implies that $1-w^{p^iq}\equiv 0\pmod{p^{i+1}}$, so $p\divides q$ and in fact $d\equiv 0\pmod{p^{i+1}}$, completing the inductive step.

This continues inductively, up to $d\equiv 0\pmod{p^a}$.
\end{proof}

The preceding lemma is applied iteratively to establish the following stronger form of the the same idea.

\begin{lemma}\label{lem:DivisibilityFull}
Suppose $w$ has multiplicative order $c$ in $\Zk$ and $c|k$. Let $g$ denote any element of $\Zk$.  If $g\left(1-w^d\right)\equiv d\pmod{k}$, then $d\equiv 0\pmod{k}$.
\end{lemma}

\begin{proof}
Let $p$ be the smallest prime divisor of $c$, and let $c=p^am$ where $p\nmid m$.

By Lemma~\ref{lem:DivisibilitySPF}, $d\equiv0\pmod{p^a}$.  Let $d=p^a\cdot e$, and let $v=w^{p^a}$, so that $v$ has multiplicative order $m$ in $\Zk$.

The main hypothesis, reduced mod $m$, says that $g\left(1-w^d\right)\equiv d\pmod{m}$, which implies $g\left(1-v^e\right)\equiv d\pmod{m}$, which implies $gp^{-a}\left(1-v^e\right)\equiv e\pmod{m}$.

Now let $q$ be the smallest prime divisor of $m$, and say $m=q^b\cdot r$ where $q\nmid r$.  Lemma~\ref{lem:DivisibilitySPF} applied to the previous congruence implies that $e\equiv 0\pmod{q^b}$ hence also $d\equiv 0\pmod{q^b}$.

The preceding steps can be applied successively to all the prime divisors of $c$ in increasing order.  If the prime factorization of $c$ is $c=p_1^{a_1}\times\cdots p_r^{a_r}$, then we conclude $d\equiv 0\pmod{p_i^{a_i}}$ for every $i$.  By the Chinese Remainder Theorem, it follows that $d\equiv 0\pmod{c}$.  And since $c$ is the multiplicative order of $w$ in $\Zk$, it follows from the main hypothesis that $d\equiv 0\pmod{k}$.
\end{proof}

\begin{lemma}\label{lem:Bijectivity}
Suppose $w$ has multiplicative order $c$ in $\Zk$, and $c\divides k$.  Let $g$ be any element of $\Zk$.  The function $f:\Zk\to\Zk$ by $f\left(x\right)=g\cdot w^x+x$ is a bijection.
\end{lemma}

\begin{proof}
Suppose $f(x)=f(y)$ for some $x$ and $y$ in $\Zk$.  That means
$$gw^x+x\equiv gw^y+y\pmod{k}$$
by definition of $f$, or equivalently,
$$gw^x\left(1-w^{y-x}\right)\equiv y-x\pmod{k}.$$
Lemma~\ref{lem:DivisibilityFull} applies and allows us to conclude $y-x\equiv0\pmod{k}$.
\end{proof}

Having established Lemmas~\ref{lem:SmallestPrime}--\ref{lem:Bijectivity}, it is easy to show that we have in fact defined a family of central digraphs, as asserted in Theorem~\ref{thm:CorrectConstruction}.

\begin{proof}[Proof of Theorem~\ref{thm:CorrectConstruction}]
Let $k$ be a positive integer, $w$ an element of multiplicative order $c$ in $\Zk$, with $c\divides k$, and let $G$ denote the graph $G(k,c,w)$ from Definition~\ref{def:GraphDef}.  Let $(p,i)$ and $(q,j)$ be any elements in $\Zk\times \Zk$, the vertex set of $G$.  Our goal is to show that there is a unique link from $(p,i)$ to $(q,j)$.

By definition of $G$, the outgoing edges from $(p,i)$ go to the vertices $\left(p+w^x i,x\right)$, where $x$ ranges over $\Zk$.  Similarly, the incoming edges into $(q,j)$ come from the vertices $\left(q-w^j x, x\right)$, where $x$ ranges over $\Zk$.

So, there is a unique link from $(p,i)$ to $(q,j)$ in $G$ if and only if there is a unique $x$ in $\Zk$ satisfying $p+w^x i=q-w^j x$ in $\Zk$.  Rearranging, there must be a unique $x$ satisfying
$$\left(i\cdot w^{-j}\right)w^x + x = (q-p)w^{-j},$$
but this is true by the bijectivity established in Lemma~\ref{lem:Bijectivity}, viewing the lefthand side as a function of $x$.
\end{proof}

Next, a simple counting method suffices to establish that different parameters give rise to non-isomorphic graphs.

\begin{proof}[Proof of Theorem~\ref{thm:DistinctGraphs}]
Let $G=G(k,c,w)$ be a \gname.  For a vertex $(p,i)$ in $G$, refer to the first component $p$ as the \textit{block number} of the vertex, and the second component $i$ as the \textit{index number}.  Recall that the idempotent vertices are those of the form $(p,0)$, so all the out-neighbors of an idempotent vertex have the same block number as the idempotent itself. We will refer to the set consisting of an idempotent together with all its out-neighbors as a \textit{block} of $G$.  Clearly, a graph isomorphism between two \gnames\ must carry idempotents to idempotents, hence also blocks to blocks.

Let the \textit{block count} of vertex $(p,i)$ be the number of distinct block numbers that occur among the out-neighbors of $(p,i)$.  Now, the out-neighbors of $(p,i)$ have block numbers of the form $p+w^ji$ where $j$ varies over $\Zk$.  That means the block count of $(p,i)$ is the number of elements in $i$'s cycle, when $w$ acts on $\Zk$ by multiplication.  Therefore, if we choose any block in $G$ and list the block counts of its vertices, we can determine the cycle type for multiplication by $w$.  The number of cycles of any given length $l$ is the number of times that $l$ appears in the list, divided by $l$.

Consider another \gname\ $G'$ which is isomorphic to $G$.  Since the isomorphism carries blocks to blocks, listing block counts for any block in $G'$ gives us the same list we found in $G$.  Hence, multiplication by the unit in $G'$ has the same cycle type as multiplication by $w$.  In particular, the two units have the same multiplicative order.
\end{proof}

\section{Automorphisms of \gnames}

\begin{proof}[Proof of Theorem~\ref{thm:AutosInSk}]
Let $G=G(k,c,w)$ be a \gname\ and $\phi:G\to G$ a graph automorphism.  As $\phi$ must permute the vertices of the form $(p,0)$ among themselves, let $\pi$ be the related permutation of $\Zk$ defined by $\pi(p)=q$ whenever $\phi(p,0)=(q,0)$.  

Now, every vertex $(p,i$) in $G$ happens to be the unique link from one idempotent to another; namely, $G$ has edges $(p,0)\to(p,i)$ and $(p,i)\to(p+i,0)$.  Thus, $\phi(p,i)$ must be the unique link between $\phi(p,0)$ and $\phi(p+i,0)$, so $\phi$ is entirely determined by its action on the idempotents.  In fact, we can express $\phi$ explicitly in terms of the permutation $\pi$, as $$\phi(p,i)=\left(\pi(p),\pi(p+i)-\pi(p)\right).$$

In the situation above, let us say that $\pi$ is the idempotent permutation determined by the automorphism $\phi$, and that $\phi$ is the automorphism induced by permutation $\pi$.  Since $\phi$ is fully determined by its idempotent permutation, the mapping sending each automorphism in $\Aut(G)$ to its corresponding idempotent permutation is injective.

Moreover, suppose $\phi_1$ and $\phi_2$ are two automorphisms of $G$, with $\pi_1$ and $\pi_2$ as their respective idempotent permutations.  Then $\phi_1\circ\phi_2(p,0) = \phi_1\left(\pi_2(p),0\right) = \left(\pi_1\circ\pi_2(p),0\right)$, so the mapping sending each automorphism in $\Aut(G)$ to its corresponding idempotent permutation is a group homomorphism from $\Aut(G)$ to the symmetric group on the $k$ elements of $\Zk$.  Thus, $\Aut(G)$ is isomorphic to a subgroup of $S_k$.
\end{proof}

In the preceding proof, we use the observation that every vertex in $G(k,c,w)$ is the unique link between a pair of idempotents.  This is a special kind of regularity which does not hold in central digraphs in general.  Example $A_4$ in \S7 of Knuth~\cite{knuth1970notes} has a vertex with no idempotents among its out-neighbors, and example $A_5$ has two such vertices.  Both graphs have a trivial automorphism group.

Automorphisms of the graphs with multiplicative order $1$ are easy to understand.

\begin{proof}[Proof of Theorem~\ref{thm:AutoGroup1}]
Let $G=G(k,1,1)$ and let $\pi$ be any permutation of $\Zk$.  We must show that the induced function
$$\phi(p,i) = \left(\pi(p),\pi(p+i)-\pi(p)\right)$$
on the vertices of $G$ is a graph automorphism.  First of all, it is easy to check that the function
$\lambda(p,i) = \left(\pi^{-1}(p),\pi^{-1}(p+i)-\pi^{-1}(p)\right)$
is an inverse for $\phi$, which is therefore a bijection.

Since the unit in $G$ is $w=1$, every edge has the form $(p,i)\to(p+i,j)$.  Applying $\phi$ to the endpoints of the edge, 
$$\phi(p,i) = (\pi(p),\pi(p+i)-\pi(p))$$
and
$$\phi(p+i,j) = \left(\pi(p+i),\pi(p+i+j)-\pi(p+i)\right),$$
and of course $\pi(p)+\left[\pi(p+i)-\pi(p)\right]=\pi(p+i)$, so $\phi(p,i)\to\phi(p+i,j)$ is an edge in $G$.  This shows that every permutation of $\Zk$ induces an automorphism of $G$ and so $\Aut(G)$ is isomorphic to (all of) $S_k$.
\end{proof}

For the remainder of the article, we will use similar (but more complicated) arguments to derive conditions which ensure that a permutation $\pi:\Zk\to\Zk$ induces a graph automorphism of $G(k,c,w)$, for higher values of $c$.  Suppose $\pi$ is such a permutation, and consider the induced function 
$$\phi(p,i) = \left(\pi(p),\pi(p+i)-\pi(p)\right)$$
on the vertices of $G$.
The edges in $G(k,c,w)$ all have the form
$$\left(a,i\right)\to \left(a+w^ji,j\right).$$
Applying $\phi$ to the endpoints of the edge,
$$\phi(a,i) = \left(\pi(a),\pi(a+i)-\pi(a)\right)$$
and
$$\phi\left(a+w^ji,j\right) = \left(\pi\left(a+w^ji\right),\pi\left(a+w^ji+j\right)-\pi\left(a+w^ji\right)\right).$$
Thus, the function $\phi$ induced by permutation $\pi$ is a graph isomorphism if and only if
$$\pi\left(a+w^ji\right)=\pi(a)+w^{\pi\left(a+w^ji+j\right)-\pi\left(a+w^ji\right)}\left[\pi\left(a+i\right)-\pi(a)\right]$$
for every $a, i,$ and $j$ in $\Zk$.  We will refer to the equation above as equation $\Iso(a,i,j)$.

We will begin by showing that every $G(k,c,w)$ has a family of automorphisms which act as affine transformations on the idempotents.

\begin{proposition}\label{prop:AffineAutomorphisms}
Let $G=G(k,c,w)$, and let $a,b\in \Zk$.  The function $\pi:\Zk\to\Zk$ by $\pi(x)=mx+b$ induces an automorphism of $G$ if and only if (1) $m$ is a unit (mod $k$), and (2) $m\equiv 1\pmod{c}$.  
\end{proposition}
Note that the congruence hypothesis makes sense, since $c|k$.
\begin{proof}
Clearly $\pi$ is a bijection if and only if $m$ is a unit, so assume $m$ is a unit.  Using the definition $\pi(x)=mx+b$, the isomorphism equation $\Iso(a,i,j)$ reduces to just
$$w^ji = w^{mj}i,$$
for all $i$ and $j$ in $\Zk$.  

If $\pi$ induces a graph automorphism, then the equation is true for $i=j=1$ in particular, which means $w = w^m$ in $\Zk$, so $m\equiv 1\pmod{c}$ since $c$ is the multiplicative order of $w$.

On the other hand, if $m\equiv 1\pmod{c}$, then $w^m = w$ in $\Zk$, so the equation is satisfied for all $i$ and $j$, and $\pi$ induces a graph automorphism.
\end{proof}

There are $\eulerphi(k)/\eulerphi(c)$ units (mod $k$) which are congruent to $1$ (mod $c$), so the preceding proposition gives a subgroup of size $k\cdot\eulerphi(k)/\eulerphi(c)$ within $\Aut(G)$.  

\section{Periodicity and Truncation}
Here, we will establish the existence of a nontrivial homomorphism from $\Aut\left(G\left(k,c,w\right)\right)$ to the symmetric group $S_c$, which we will call the \textit{truncation} homomorphism.  The kernel of this homomorphism gives us an interesting normal subgroup.

\begin{proposition}\label{prop:Periodicity}
Let $G=G(k,c,w)$ and suppose $\pi$ is a permutation of $\Zk$ which induces an automorphism of $G$.  Then $\pi(i+c)\equiv \pi(i)\pmod{c}$ for every $i$ in $\Zk$.
\end{proposition}
\begin{proof}
Let $i\in \Zk$.  For any $p\in \Zk$, isomorphism equation $\Iso(p,i-p,c)$ says that
$$\pi(i)=\pi(p)+w^{\pi(i+c)-\pi(i)}\left[\pi(i)-\pi(p)\right].$$
Since $\pi$ is a permutation of $\Zk$, there is a particular $p\in \Zk$ for which $\pi(i)-\pi(p)=1$. Evaluating the above for that value of $p$ gives,
$$1 = w^{\pi(i+c)-\pi(i)},$$
which implies $\pi(i+c)\equiv\pi(i)\pmod{c}$, since $c$ is the multiplicative order of $w$.
\end{proof}

Proposition~\ref{prop:Periodicity} says that permutations $\pi$ which induce an automorphism of $G(k,c,w)$ are ``periodic, with period $c$, when considered mod $c$.''  This means that $\pi$ is a lifting of some permutation on $\Zmod{c}$ which we will call the \textit{truncation} of $\pi$, denoted $T(\pi)$.

\begin{example}
Consider $G=G(9,3,4)$ and the following permutation of $\Zmod9$, expressed in tabular form:
$$\pi=\begin{pmatrix}
0 & 1 & 2 & 3 & 4 & 5 & 6 & 7 & 8 \\
4 & 5 & 3 & 7 & 8 & 6 & 1 & 2 & 0 \\
\end{pmatrix}.$$
This $\pi$ exhibits the periodicity described in Proposition~\ref{prop:Periodicity} (and does, in fact, induce a permutation of $G$, although we do not need to verify that).  Reducing all the entries mod $3$, the truncation $T(\pi)$ acts as a permutation of $\Zmod3$, namely,
$$T(\pi)=\begin{pmatrix}
0 & 1 & 2 & 0 & 1 & 2 & 0 & 1 & 2 \\
1 & 2 & 0 & 1 & 2 & 0 & 1 & 2 & 0 \\
\end{pmatrix}$$
or, more normally written,
$$T(\pi)=\begin{pmatrix}
0 & 1 & 2 \\
1 & 2 & 0 \\
\end{pmatrix}.
$$
\end{example}

The operation of truncation commutes with composition, so $T$ defines a homomorphism from $\Aut\left(G(k,c,w)\right)$ to the symmetric group $S_c$.  By Proposition~\ref{prop:AffineAutomorphisms}, the function $\pi:\Zk\to\Zk$ by $\pi(x)=x+1$ induces an automorphism of $G$, and $T(\pi)$ is a $c$-cycle in $S_c$.  This shows that $T: \Aut\left(G(k,c,w)\right)\to S_c$ is a nontrivial homomorphism, when $c>1$.   

\section{Automorphisms of graphs with unit $w=-1$}

\begin{proof}[Proof of Theorem~\ref{thm:AutoGroup2}]  Let $k$ be a positive, even integer, and let $\pi:\Zk\to\Zk$ be an idempotent mapping which induces an automorphism of $G(k,2,-1)$.  We aim to show that $\pi$ must be one of the affine functions described in Proposition~\ref{prop:AffineAutomorphisms}.  Let $m=\pi(1)-\pi(0)$ and let $b=\pi(0)$.  Note that $\pi(0)=0m+b$ and $\pi(1)=1m+b$.

We know that $\pi$ satisfies the general isomorphism equation (using $w=-1$),
$$\pi\left(a+\left(-1\right)^ji\right)=\pi(a)+\left(-1\right)^{\pi\left(a+\left(-1\right)^ji+j\right)-\pi\left(a+\left(-1\right)^ji\right)}\left[\pi\left(a+i\right)-\pi(a)\right]$$
for all $a, i,$ and $j$ in $\Zk$.  Note, though, that by Periodicity (Proposition~\ref{prop:Periodicity}), the parity of the exponent
$$ {\pi\left(a+\left(-1\right)^ji+j\right)-\pi\left(a+\left(-1\right)^ji\right)} $$
is simply the parity of $j$.  Therefore, the general equation is equivalent to the much simpler form
$$\pi\left(a+\left(-1\right)^ji\right)=\pi(a)+\left(-1\right)^{j}\left[\pi\left(a+i\right)-\pi(a)\right]$$
for all $a, i,$ and $j$ in $\Zk$.  Specializing to $j=1$ and $i=1$, it follows that
$$\pi(a+1) = \pi(a) + \left(\pi(a)-\pi(a-1)\right)$$
for all $a$ in $\Zk$.
Applying this to $a=1$,
$$\pi(2) = m+b+m = 2m+b,$$
then applying it to $a=2$,
$$\pi(3) = 2m+b + m = 3m+b,$$
and so on, showing that $\pi(x)=mx+b$ for each $x\in\Zk$.  Thus the full automorphism group of $G(k,2,-1)$ is isomorphic to the group of affine functions from $\Zk$ to itself.
\end{proof}

\section{The kernel of truncation for graphs of prime-squared size}

Throughout this section, let $p$ be a prime, and let $G=G(p^2,p,1+p)$.  That is, $G$ has $p^4$ vertices and has multiplicative order $p$.  Let $w=1+p$.  It follows from the binomial theorem that $w^x = 1+px$ in $\Zpp$, and in particular, $w$ does have multiplicative order $p$ in $\Zpp$.

\begin{proposition}\label{prop:Kernel1}
Suppose $\pi:\Zpp\to\Zpp$ is a permutation satisfying
\begin{enumerate}
\item $\pi(j)\equiv j\pmod{p}$ for all $j\in\Zpp$, and
\item $\pi(a+p)\equiv \pi(a)+p\pmod{p^2}$ for all $a$.
\end{enumerate}
Then $\pi$ induces an automorphism of $G$.
\end{proposition}

\begin{proof}
We need to show that the general isomorphism equation 
$$\pi\left(a+w^ji\right)=\pi(a)+w^{\pi\left(a+w^ji+j\right)-\pi\left(a+w^ji\right)}\left[\pi\left(a+i\right)-\pi(a)\right]$$
is satisfied for all $a,i,$ and $j$ in $\Zpp$.  Note that, for any $x$ in $\Zpp$, $w^{\pi(x)} = w^x$ by hypothesis (1).
Evaluating the lefthand side first,
\begin{align*}
\pi\left(a + w^j i\right)
  &= \pi\left(a + i + i j p\right)
     &&\text{since } w^j = 1 + j p,\\
  &= \pi(a + i) + i j p
     &&\text{by hypothesis (2),}
\end{align*}
with the $=$ sign denoting equality in $\Zpp$.

On the righthand side, the exponential expression simplifies as follows:
\begin{align*}
w^{\pi\left(a+w^ji+j\right)-\pi\left(a+w^ji\right)} 
    &= w^{\pi\left(a+w^ji+j\right)}\cdot w^{-\pi\left(a+w^ji\right)} \\
    &= w^{(a+i+j+ijp)}\cdot w^{-(a+i+ijp)} \\
    &= \left(1+ap+ip+jp\right)\left(1-ap-ip\right) \\
    &= 1+jp
\end{align*}
Also, $\pi(a+i)-\pi(a) = \pi(i) + lp$ for some integer $l$, by hypothesis (1).  So the entire righthand side of the isomorphism equation simplifies as
\begin{align*}
\pi(a)+\left(1+jp\right)\left[\pi(a+i) - \pi(a)\right] 
    &= \pi\left(a+i\right)+jp\left[\pi(a+i)-\pi(a)\right] \\
    &= \pi\left(a+i\right)+jp\left[i+lp\right] \\
    &= \pi\left(a+i\right)+ijp.
\end{align*}
This agrees with the lefthand side for all $a,i,$ and $j$.
\end{proof}

Note that hypothesis (1) precisely characterizes permutations in the kernel of the truncation homomorphism $T$ in this case.  The following proposition shows that permutations in the kernel of $T$ also satisfy hypothesis (2).

\begin{proposition}\label{prop:Kernel2}
Let $\pi$ be a permutation of $\Zpp$ which induces an automorphism of $G$.  If $\pi\in \ker T$ then $\pi(a+p) = \pi(a)+p$ for all $a$ in $\Zpp$.
\end{proposition}

\begin{proof}
Since $\pi$ induces an automorphism, $\pi$ satisfies every instance of the isomorphism equation.  In particular, equation $\Iso(a-1,1,1)$ says that 
\begin{align*}
\pi(a+p) 
    &= \pi(a-1) + w^{\pi(a+p+1)-\pi(a+p)}\left[\pi(a)-\pi(a-1)\right] \\
    &= \pi(a-1) + (1+ap+p)(1-ap)\left[\pi(a)-\pi(a-1)\right] \\
    &= \pi(a-1) + \pi(a) -\pi(a-1) +p\pi(a) -p\pi(a-1) \\
    &= \pi(a) + pa - p(a-1) 
	&\text{since }\pi\in\ker T\\
    &= \pi(a)+p
\end{align*}
in $\Zpp$, as claimed.
\end{proof}

\begin{proposition}\label{prop:Kernel3}
For $G=G(p^2,p,1+p)$, the kernel of the truncation homomorphism $T$ is abelian.
\end{proposition}
\begin{proof}
Let $\alpha, \beta$ be permutations of $\Zpp$ which belong to $\ker T$, and let $x\in\Zpp$.
Then $\beta(x) \equiv x\pmod{p}$ since $\beta\in \ker T$, so $\beta(x)=x+lp$ for some integer $l$.
Similarly, $\alpha(x)=x+rp$ for some integer $r$.  Thus,
\begin{align*}
\alpha\circ\beta(x) &= \alpha(x+lp) \\
    &= \alpha(x) + lp &\text{by Proposition~\ref{prop:Kernel2}} \\
    &= x + rp + lp,
\end{align*}
and
\begin{align*}
\beta\circ\alpha(x) &= \beta(x+rp) \\
    &= \beta(x) + rp &\text{by Proposition~\ref{prop:Kernel2}} \\
    &= x + lp + rp.
\end{align*}
This shows $\alpha\circ\beta = \beta\circ\alpha$, as claimed.
\end{proof}

\begin{proposition}\label{prop:KernelOfT}
Let $G=G\left(p^2,p,1+p\right)$, where $p$ is prime.  The kernel of the truncation homomorphism $T:\Aut(G)\to S_p$ is isomorphic to $\left(\Zmod{p}\right)^p$, the elementary abelian $p$-group of order $p^p$.
\end{proposition}
\begin{proof}
If $\pi\in\ker T$, then $\pi$ is totally determined by the values of $\pi(0),\ldots,\pi(p-1)$, by Proposition~\ref{prop:Kernel2}.  Moreover, $\pi\in\ker T$ means that $\pi(j)\equiv j\pmod{p}$ for every $j$, so there are at most $p$ possibilities for $\pi(j)$, for each $j\in\left\{0,\ldots,p-1\right\}$.  By Proposition~\ref{prop:Kernel1}, every possible choice of those values corresponds to an actual automorphism of $G$.  Thus, $\ker T$ has $p^p$ elements.

Finally, by Proposition~\ref{prop:Kernel2}, $\pi\in \ker T$ implies $\pi^p(a) = a+p\cdot p$, which equals $a$ in $\Zpp$, so $\pi^p$ is the identity for every $\pi\in \ker T$.  By Proposition~\ref{prop:Kernel3}, $\ker T$ is abelian, so in fact $\ker T$ is the elementary abelian $p$-group of order $p^p$.
\end{proof}

\section{The image of truncation for graphs of prime-squared size}
Throughout this section, again, let $p$ be a prime, and let $G=G(p^2,p,1+p)$.  To complete the analysis of the automorphism group of $G$, we will study the image of the truncation homomorphism $T:\Aut(G)\to S_p$.

\begin{proposition}\label{prop:TwoValues}
Let $\pi$ be a permutation of $\Zmod{p^2}$ which induces an automorphism of $G$.  The $\pi(a)$ is determined (mod $p$), for all $a\in \Zmod{p^2}$, by the values of $\pi(0)$ and $\pi(1)$ (mod $p$).
\end{proposition}

\begin{proof}
Let $a\in \Zmod{p^2}$, and let $w$ stand for $1+p$. Equation $\Iso\left(a, -aw^{-1}, 1\right)$ says 
$$ \pi(0) = \pi(a) + w^{\pi(1)-\pi(0)}\left[\pi(pa)-\pi(a)\right]. $$
Rearranging,
$$ \pi(a) - w^{\pi(1)-\pi(0)}\pi(a) = \pi(0) - w^{\pi(1)-\pi(0)}\pi(pa)$$
so
$$ p\left(\pi(1)-\pi(0)\right)\pi(a) = \pi(0) - w^{\pi(1)-\pi(0)}\pi(pa).$$
Let $R$ denote the value of the righthand side.  By Proposition~\ref{prop:Periodicity}, the value of $\pi(pa)$ is determined (mod $p$) by the value of $\pi(0)$.  Also, since the lefthand side is a multiple of $p$, $R$ must be as well.  Cancelling a $p$ from both sides,
$$ \left(\pi(1)-\pi(0)\right)\pi(a) \equiv R/p \pmod{p} $$
Also by Proposition~\ref{prop:Periodicity}, $\pi(0)$ and $\pi(1)$ cannot be equal (mod $p$).  So $\pi(1)-\pi(0)$ is invertible (mod $p$), and the previous equation implies that $\pi(a)$ is determined (mod $p$) by $R$, which is determined by the values of $\pi(0)$ and $\pi(1)$ (mod $p$).
\end{proof}

\begin{corollary}\label{cor:Image1}
Let $G=G\left(p^2,p,1+p\right)$ and let $\pi$ be a permutation of $\Zmod{p^2}$ which induces an automorphism of $G$.  Then $T(\pi)$, the truncation of $\pi$ is determined by the values of $\pi(0)$ and $\pi(1)$ (mod $p$).  In particular, $|\im T| \le p(p-1)$.
\end{corollary}

Next, we will show that there are in fact $p(p-1)$ distinct elements in the image of $T$.  To do this, we will construct certain permutations of $\Zmod{p^2}$ and show that they induce automorphisms of $G$.  In defining these permutations, it is helpful to identify the set $\Zmod{p^2}$ with the cartesian product $\Zmod{p}\times \Zmod{p}$ by associating each $a\in \Zmod{p^2}$ with the ordered pair containing its quotient and remainder (mod $p$).  

{\em Example (using $p=5$).} We identify the element $22$ in $\Zmod{25}$ with the ordered pair $(4,2)$ in $\Zmod{5}\times\Zmod{5}$, as $22=4\cdot 5+2$.

To construct our permutations of $\Zmod{p^2}$, we will begin with an affine function $mx+b$ from $\Zmod{p}$ to itself, where $m$ is a unit in $\Zmod{p}$ and $b$ is any element of $\Zmod{p}$.  Define $\pi:\Zmod{p^2}\to\Zmod{p^2}$ by identifying each element of $\Zmod{p^2}$ with its quotient/remainder pair in $\Zmod{p}\times \Zmod{p}$, and setting $\pi(q,r) = \left(m^2q, mr+b\right)$, using $\Zmod{p}$ arithmetic in each component.

{\em Example.}  Beginning with the function $3x+2$ on $\Zmod{5}$, we define the function $\pi$ by $\pi(q,r) = \left(3^2q,3r+2\right)$.  To apply this to an element of $\Zmod{25}$, we first separate it into its quotient and remainder parts.
\par\qquad For example, $\pi(22) = \pi(4,2) = \left(3^2\cdot 4, 3\cdot2+2\right) = \left(1,3\right) = 8$.  
Here is a complete table of values for this particular function $\pi$.

$$ \begin{array}{|c|c|c|c|c|c|c|c|c|c|c|c|c|c|}
\hline
 a & 0 & 1 & 2 & 3 & 4 & 5 & 6 & 7 & 8 & 9 & 10 & 11 & 12 \\
\hline
 \pi(a) & 2 & 0 & 3 & 1 & 4 & 22 & 20 & 23 & 21 & 24 & 17 & 15 & 18 \\
\hline
\end{array} $$

$$ \begin{array}{|c|c|c|c|c|c|c|c|c|c|c|c|c|}
\hline
 a & 13 & 14 & 15 & 16 & 17 & 18 & 19 & 20 & 21 & 22 & 23 & 24 \\
\hline
 \pi(a) & 16 & 19 & 12 & 10 & 13 & 11 & 14 & 7 & 5 & 8 & 6 & 9 \\
\hline
\end{array} $$

\begin{proposition}\label{prop:ComplementaryAutomorphisms}
Let $G=G\left(p^2,p,1+p\right)$.  Let $m$ be a unit in $\Zmod{p}$ and let $b$ be any element of $\Zmod{p}$.  Identifying $\Zmod{p^2}$ (as a set) with $\Zmod{p}\times \Zmod{p}$, as described above, define $\pi:\Zmod{p^2}\to \Zmod{p^2}$ by $\pi\left(q,r\right) = \left(m^2q, mr+b\right)$.  Then $\pi$ induces an automorphism of $G$.
\end{proposition}

\begin{proof}
For $a\in \Zmod{p^2}$, let $\quo(a)$ denote the quotient of $a$ by $p$, and let $\rem(a)$ denote the remainder, so that $a$ is identified with the pair $\left(\quo(a),\rem(a)\right)$.

Using $w$ to denote the unit $1+p$, we need to show that equation $\Iso(a,i,j)$, namely
$$\pi\left(a+w^ji\right)=\pi(a)+w^{\pi\left(a+w^ji+j\right)-\pi\left(a+w^ji\right)}\left[\pi\left(a+i\right)-\pi(a)\right],$$
is satisfied for all $a$, $i$, and $j$ in $\Zmod{p^2}$.   We will study the quotient and remainder parts of the two sides separately.

In fact, since $w\equiv 1\pmod{p}$ and $\pi$ acts as an affine function on the remainder parts, it is easy to see that both sides leave the same remainder (mod $p$).  

To understand the quotient parts, let us introduce two functions from $\Zmod{p}\times \Zmod{p}$ to $\Zmod{p}$ which allow us to express $\Zmod{p^2}$ addition and subtraction in terms of ordered pairs.
Namely, define
$$\car(a,i) := \quo(a+i) - \left[\quo(a)+\quo(i)\right]$$
and
$$\bor(a,i) := \quo(a-i) - \left[\quo(a)-\quo(i)\right].$$
These are defined so that
$$ \quo(a+i) = \quo(a) + \quo(i) + \car(a,i) $$
and
$$ \quo(a-i) = \quo(a) - \quo(i) + \bor(a,i). $$
But note that, for any $x$ and $y$ in $\Zmod{p}$,
\begin{align*}
\bor(x+y, x) &= \quo(y) - \left[\quo(x+y) - \quo(x)\right] \\
& = \quo(x) + \quo(y) - \quo(x+y) \\
& = -\car(x,y),
\end{align*}
so that we always have
\begin{equation}\label{eqn:borcar}
\bor(x+y, x) + \car(x,y) = 0. 
\end{equation}

Looking first at the lefthand side of equation $\Iso(a,i,j)$,
\begin{equation}\label{eqn:lhs}
\quo\left(\pi\left(a+w^ji\right)\right) = m^2\left(\quo(a)+\quo(i)+ij+\car(a,i)\right).
\end{equation}

We will build up the righthand side piece by piece. The intermediate results are expressed as ordered quotient-and-remainder pairs. First,
\begin{equation}\label{eqn:mi2}
w^{\pi\left(a+w^ji+j\right)-\pi\left(a+w^ji\right)} = \left(mj, 1\right)
\end{equation}
and
\begin{equation}\label{eqn:mi3}
\pi(a+i)-\pi(a) = 
\left(
    m^2\left(\quo(i)+\car(a,i)\right) + \bor\left(m\left(a+i\right)+b,ma+b\right),
    \rem(mi)
\right).
\end{equation}

Multiplying (\ref{eqn:mi2}) and (\ref{eqn:mi3}) together gives

\begin{equation*}
\left(
m^2\left(\quo(i)+\car(a,i)+ij\right)+\bor\left(m(a+i)+b, ma+b\right),
\rem(mi)
\right).
\end{equation*}

So, in all, the quotient part of the righthand side of equation $\Iso(a,i,j)$ is
\begin{equation}\label{eqn:rhs}
m^2\left(\quo(a)+\quo(i)+\car(a,i)+ij\right) + \bor\left(m(a+i)+b,ma+b\right) + \car\left(ma+b, mi\right).
\end{equation}

Comparing the lefthand side result (\ref{eqn:lhs}) with the righthand side result (\ref{eqn:rhs}),
all of the terms with a factor of $m^2$ are common to both sides, and the remaining expression 
$\bor(ma+b+mi,ma+b) + \car(ma+b,mi)$ which appears in (\ref{eqn:rhs}) is always zero by the 
borrow-and-carry identity (\ref{eqn:borcar}).  Thus, equation $\Iso(a,i,j)$ is satisfied for 
all $a$, $i$, and $j$, and $\pi$ induces a graph automorphism. 
\end{proof}

Now we can complete the analysis of the automorphism group of $G(p^2,p,1+p)$.

\begin{proof}[Proof of Theorem~\ref{thm:AutoGroup3}].  Let $G=G(p^2,p,1+p)$ and let $A$ denote the set of automorphisms of $G$ of the form
$(q,r)\mapsto \left(m^2q,mq+r\right)$ as described in Proposition~\ref{prop:ComplementaryAutomorphisms}.  

It is straightforward that $A$ is a subgroup of $\Aut(G)$ and that $A\cong \text{Aff}\left(\FieldF_p\right)$.  In particular, $|A|=p(p-1)$.

Proposition~\ref{prop:KernelOfT} showed that $\ker T$ is the elementary abelian $p$-group of order $p^p$.

From Proposition~\ref{prop:TwoValues}, it follows that $A\cap \left(\ker T\right)$ contains only the identity automorphism.

Since we know from Corollary~\ref{cor:Image1} that $|G/\ker T|\le p(p-1)$, it follows that $A$ is a complement to $\ker T$ in $\Aut(G)$,
and $\Aut(G)\cong \left(\ker T\right) \rtimes A$.  In other words, $\Aut(G)$ has a normal elementary abelian $p$-group of order $p^p$
and a complementary $\text{Aff}\left(\FieldF_p\right)$.  
\end{proof}

\section{Open Questions}
If a reader is interested in investigating this family of digraphs further, I suggest a few problems which seem approachable at this point:
\begin{enumerate}
\item Establish the converse of Theorem~\ref{thm:DistinctGraphs} -- that is, briefly, that units with the same cycle type give rise to isomorphic graphs.  An easy case to start is when the size $k$ is a prime power, but I haven't been able to complete the proof for arbitrary $k$.

\item Compute the rank of the adjacency matrix in terms of the parameters $(k,c,w)$.  The possible ranks for a central digraph are analyzed in \cite{shader1974221} and subsequently \cite{curtis2004central}.  It would be interesting to establish a concrete relationship between the size of the automorphism group and the matrix rank.

\item It was conjectured \cite{curtis2004central} that every central digraph of a given size can be obtained from the natural one by switching operations.  This has been shown false in general \cite{kundgen2011switchings}, but one might ask whether all \gnames\ can be obtained in this way.

\item Much of the analysis of the automorphism groups of \gnames\ here is based on the observation that every vertex is a link between two idempotent vertices.  How common or exceptional is this property among central digraphs in general?

\end{enumerate}


\bibliographystyle{amsplain}
\bibliography{plain-central-digraphs}

\end{document}